\documentclass[11 pt, reqno]{amsart}
\usepackage{amsfonts}
\usepackage{amsthm}
\usepackage{amssymb}
\usepackage{latexsym}
\usepackage{amsmath}
\usepackage{mathrsfs}
\usepackage{dsfont}
\usepackage{a4wide}
\usepackage[usenames,dvipsnames]{xcolor}

\usepackage[all]{xy}
\usepackage{graphicx}

\theoremstyle{plain}
\newtheorem{theorem}{Theorem}
\newtheorem{tw}[theorem]{Theorem}
\newtheorem{lem}[theorem]{Lemma}
\newtheorem{propn}[theorem]{Proposition}
\newtheorem{cor}[theorem]{Corollary}

\theoremstyle{definition}

\newcommand{\bc}{\mathbb{C}}
\newcommand{\br}{\mathbb{R}}

\newcommand\conv{\star}

\newcommand{\Nphi}{\mathcal{N}_\phi}
\newcommand{\inv}{^{-1}}

\newcommand\Lone{{L}^1}
\newcommand\Lonesharp{{L}_\sharp\sp1}
\newcommand\Ltwo{{L}^2}
\newcommand\Linf{{L}^\infty}
\let\C=C
\newcommand{\M}{\mathit{M}}
\newcommand{\G}{\mathbb{G}}

\newcommand{\ot}{\otimes}

\newcommand{\id}{\mathrm{id}}
\newcommand{\Com}{\Delta}
\newcommand{\blg}{\mathsf{B}}
\newcommand{\Clg}{\mathsf{C}}
\newcommand{\Dlg}{\mathsf{D}}
\newcommand{\WW}{{\mathds{V}\!\!\text{\reflectbox{$\mathds{V}$}}}}
\newcommand{\Ww}{\mathds{W}}
\newcommand{\ww}{\mathrm{W}}
\newcommand{\QG}{\G}
\newcommand{\QH}{\mathbb{H}}
\newcommand{\hQG}{\widehat{\QG}}

\newcommand{\Pu}{\textup{Prob}^u}
\newcommand\wot{\mathop{\overline\otimes}}

\newenvironment{rlist}
{

\begin{enumerate}}
{\end{enumerate}}

\begin{document}

\title{Idempotent states on locally compact quantum groups II}

\keywords{Locally compact quantum group, idempotent state,
invariant subalgebra}

\subjclass[2000]{Primary 46L65, Secondary 43A05, 46L30, 60B15}

\begin{abstract}
  \noindent
 Correspondence between idempotent states and expected right-invariant
 subalgebras is extended to non-coamenable, non-unimodular locally
 compact quantum groups; in particular left convolution operators are
 shown to automatically preserve the right Haar weight.
\end{abstract}

\author{Pekka Salmi}
\address{Department of Mathematical Sciences,
University of Oulu, PL 3000, FI-90014 Oulun yliopisto, Finland}
\email{pekka.salmi@iki.fi}

\author{Adam Skalski}
\address{Institute of Mathematics of the Polish Academy of Sciences,
ul.~\'Sniadeckich 8, 00--656 Warszawa, Poland }
\email{a.skalski@impan.pl}

\maketitle

Idempotent states on finite and compact quantum groups, generalising
the idempotent probability measures on compact groups, have been
studied in a series of papers [FS$_{1-2}$] and \cite{FST} (see also
the survey \cite{Pekkasurvey}, where one can find the probabilistic
and harmonic-analytic motivations behind investigating such
objects). In the article \cite{SSIdem} the main results related to
idempotent states were extended to the \emph{locally compact} case,
under the assumption that the locally compact quantum group $\QG$  in
question is \emph{unimodular} and \emph{coamenable}. The first of
these properties is automatically satisfied in the compact case,
whereas the second, indeed assumed in \cite{FST} and \cite{FS},  means
that all idempotent states may be viewed as bounded functionals on the
C*-algebra $\C_0(\QG)$, and not, as is the case in general, on its
universal counterpart, $\C_0^u(\QG)$.

In this short paper we show that in fact the main results of
\cite{SSIdem} hold also when both of these assumptions are
dropped. We also connect the von Neumann algebraic picture, i.e.\ working
with the algebra $\Linf(\QG)$ of `essentially bounded functions' on
$\QG$, with the C*-algebraic one.  Thus we prove the following
theorem (the detailed explanation of the terms used below may be found
in the main body of the paper; note that expected C*-subalgebras are automatically non-zero).

\begin{tw}\label{maintheorem}
Let $\QG$ be an arbitrary locally compact quantum group. There is a one-to-one correspondence between the following objects:
\begin{rlist}
\item idempotent states in $\Pu(\QG)$;
\item right-invariant expected C*-subalgebras of $\C_0(\QG)$;
\item right-invariant expected von Neumann subalgebras of $\Linf(\QG)$;
\item left-invariant expected C*-subalgebras of $\C_0(\QG)$;
\item left-invariant expected von Neumann subalgebras of $\Linf(\QG)$.
\end{rlist}
\end{tw}

The key part in the proof of the above result is based on the
following theorem.

\begin{tw}\label{thm:weightpreserv}
Let $\omega \in \Pu(\QG)$ be an idempotent state. Then the associated
left multiplier $L_\omega$ preserves both the left and the right Haar
weight.
\end{tw}

Note one immediate corollary, which admits an obvious left-sided counterpart.

\begin{cor}\label{corollary}
A right-invariant $\psi$-expected C*-subalgebra of $\C_0(\QG)$ is
automatically expected.
Similarly, a right-invariant $\psi$-expected von Neumann subalgebra of
$\Linf(\QG)$ is automatically expected.
\end{cor}

We also deduce from the main theorem some properties enjoyed
by idempotent states.

\begin{propn}\label{scaling}
Idempotent states in $\Pu(\G)$ are invariant with respect to the
antipode $S^u$, the unitary antipode $R^u$ and the scaling group
$\{\tau_t^u: t \in \br\}$.
\end{propn}

Finally we obtain the following characterisation of Haar idempotent states, i.e.\ those that arise as Haar states on compact quantum subgroups of $\QG$.

\begin{tw} \label{Haar}
Let $\omega \in \Pu(\G)$ be an idempotent state. The following conditions are equivalent:
\begin{rlist}
\item $\omega$ is a Haar idempotent;
\item the null space $N_\omega= \{a \in \C_0^u(\QG): \omega(a^*a) = 0\}$ is a two-sided ideal;
\item the right-invariant expected C*-subalgebra of $\C_0(\QG)$
  associated to $\omega$ is symmetric.
\end{rlist}
\end{tw}

Here a remark is in place: after the first version of this article was
circulated, P.\,Kasprzak and F.\,Khosravi shared with us a draft of
their paper \cite{KK}. The correspondence (i)$\iff$(iii)
is essentially contained in that paper. Note also that Section 4 of
\cite{KK} improves  Corollary~\ref{corollary}.

The fact that the coamenability assumption can be dropped in the
compact case was also observed in \cite{HunUweAdam} -- there it is
much simpler, as one can use essentially purely algebraic methods.
The locally compact non-coamenable context neccessitates  applying
the theory of normal (left) multipliers on $\Linf(\QG)$, as developed
in \cite{jnr} and \cite{daws}. Finally we note that as some of the
proofs in the non-coamenable setting use the techniques and arguments
developed in \cite{SSIdem}, in the text below we often simply refer to
the appropriate parts of that paper.

\vspace*{0.5 cm}
\noindent
    {\bf Acknowledgement.}\
AS  was partially supported by the NCN (National Centre of Science) grant
2014/14/E/ST1/00525. We thank P.\,Kasprzak and F.\,Khosravi for
sharing with us a draft of \cite{KK} and useful comments.
We thank M. Daws for providing an argument for
Lemma~\ref{lem::C0multipliers}.

\section{Notation, terminology and background} \label{background}

Throughout the paper $\QG$ will be a locally compact quantum group in
the sense of Kustermans and Vaes, described via  the `algebra of
essentially bounded functions' $\Linf(\QG)$ and the the `algebra of
continuous functions vanishing at infinity'
$\C_0(\QG)\subset\Linf(\QG)$, where $\C_0(\QG)$ is a (usually
non-unital) C*-algebra and $\Linf(\QG)$ a von Neumann algebra equipped
with a \emph{coproduct}
$\Com:\Linf(\QG) \to \Linf(\QG) \wot \Linf(\QG)$. We will in general
use the notation and
terminology of \cite{dfsw}  and refer to the precise definitions of
all the objects appearing below to \cite{KV}, \cite{dfsw} and
\cite{SSIdem}.
The \emph{left} and the \emph{right Haar weights} of $\QG$
will be denoted by symbols $\phi$ and $\psi$, respectively.
The von Neumann algebra $\Linf(\G)$, so also $\C_0(\G)$ and
the multiplier algebra $\C_b(\G):= M(\C_0(\G))$,
act on the GNS Hilbert space $\Ltwo(\QG)$ of $\phi$.
We will use the standard notation
$\mathcal{N}_\phi=\{x \in \Linf(\QG):\phi(x^*x) < \infty\}$.  The
universal version of $\C_0(\QG)$ will be written as
$\C_0^u(\QG)$, the corresponding coproduct as $\Com_u$, and the
\emph{reducing morphism} from $\C_0^u(\QG)$ onto $\C_0(\QG)$ as
$\Lambda$. The dual space of  $\C_0^u(\G)$ will be denoted by
$\M\sp u(\QG)$, and the corresponding state space by $\Pu(\G)$ -- recall
that $\M\sp u(\G)$ is equipped with a natural Banach algebra structure
given by the convolution product
$\mu \star \nu:= (\mu \ot \nu) \circ \Com_u$ for
$\mu, \nu \in \M\sp u(\QG)$.  An element $\omega \in \Pu(\G)$ is called an
\emph{idempotent state} if $\omega \star \omega= \omega$. Also
the predual of $\Linf(\QG)$, denoted by $\Lone(\QG)$, admits a
similar convolution product; we will need at some point the fact
that $\Lone(\QG)$ admits a dense subalgebra $\Lonesharp(\QG)$, which
allows a natural involution, related to the \emph{antipode} $S$ of
$\Linf(\QG)$. The \emph{unitary antipode} of $\Linf(\QG)$ will be
denoted by $R$, the \emph{modular automorphisms of the left Haar
  weight} by $\sigma_t$ ($t \in \br$), the  \emph{scaling
  automorphisms} by $\tau_t$ ($t \in \br$) and the \emph{modular
  element of $\QG$} by $\delta$; all these have the counterparts on
the universal level, denoted respectively by $S^u$, $R^u$,
$\sigma_t^u$, $\tau_t^u$ and $\delta_u$ (see \cite{ku} for details).

Each locally compact quantum group admits a \emph{dual locally
  compact quantum group} $\hQG$. The \emph{multiplicative unitary}
belonging to the multiplier algebra  $\M(\C_0(\G) \ot \C_0(\hQG))$
implementing the coproduct will be denoted by $\ww$;
we will also use its semi-universal version
$\Ww \in \M(\C_0^u(\G) \ot \C_0(\hQG))$ and
the universal version $\WW \in \M(\C_0^u(\G) \ot\C_0^u(\hQG))$.

We say that a locally compact quantum group $\QH$ is \emph{compact} if
the algebra $\C_0(\QH)$ is unital. We then denote the respective
reduced/universal C*-algebras of `functions' on $\QH$ by $\C(\QH)$ and
$\C^u(\QH)$; each compact quantum group $\QH$ admits a \emph{Haar
  state} $h_\QH \in \Pu(\QH)$, a unique state such that $h_\QH \star
\mu = \mu \star h_\QH = h_\QH$ for all $\mu \in \Pu(\QH)$. A compact
quantum group $\QG$ is said to be a \emph{\textup{(}closed\textup{)}
  quantum subgroup}
of a locally compact quantum group $\QG$ if there exists a surjective
morphism $\pi:\C_0^u(\QG) \to \C^u(\QH)$ intertwining the respective
coproducts. In such a case it is easy to see that $\omega:=h_\QH \circ
\pi$ is an idempotent state; the idempotent states which arise in this
way are called \emph{Haar idempotents}.

A C*-subalgebra $\Clg \subset\C_0(\QG)$ is said to be
\emph{right invariant} if $(\id \ot \mu)(\Com(\Clg)) \subset\Clg$ for
all $\mu \in \C_0(\G)^*$. It is said to be $\phi$-\emph{expected}
(respectively, \emph{$\psi$-expected}) if there exists a conditional
expectation $E$ from $\C_0(\QG)$ onto $\Clg$ that is
$\phi$-preserving (respectively, $\psi$-preserving) and simply
\emph{expected} if there exists a conditional expectation $E$
onto $\Clg$ that is both $\psi$-preserving and
$\phi$-preserving; note that expected subalgebras are neccessarily non-trivial. Similarly a von Neumann subalgebra $\Dlg \subset
\Linf(\QG)$ is said to be \emph{right invariant} if
$(\id \ot\mu)(\Com(\Dlg)) \subset\Dlg$ for all $\mu \in \Lone(\G)$
(equivalently, $\Com(\Dlg) \subset \Dlg \wot \Linf(\QG) $) and we call it
\emph{$\phi$-expected} (or \emph{$\psi$-expected}, or
\emph{expected}) if the respective expectations from $\Linf(\QG)$ onto
$\Dlg$ are in addition normal. For a discussion of weight
preservation etc., we refer to \cite{SSIdem}. Finally we say that a
right-invariant C*-subalgebra $\Clg \subset\C_0(\QG)$ is
\emph{symmetric} if the following holds:
\[ \ww^*(\Clg\ot 1) \ww  \subset \M(\Clg \ot C_0(\hQG)).\]
Note that Proposition 3.1 of \cite{KK} (or rather its
right version) shows that non-zero right-invariant subalgebras of
$\C_0(\G)$ are automatically nondegenerate.

We now summarise the main facts concerning the unital, completely
positive, normal left multipliers on $\Linf(\QG)$, proved in papers
\cite{jnr} and \cite{daws} and gathered in \cite{dfsw}.

\begin{tw} \label{Matt}
Let $L: \Linf(\QG) \to \Linf(\QG)$ be a normal unital completely positive map. Then the following are equivalent:
\begin{rlist}
\item  $L=T^*$, where $T: \Lone(\QG) \to \Lone(\QG)$ is a bounded
  left module map;
\item \begin{equation}\Com \circ L = (L \ot \id) \circ \Com\label{commutCom};\end{equation}
\item there exists $a \in \C_b(\hQG)$ such that
      $(L \ot \id) (\ww) = (1 \ot a) \ww$;
\item there exists $\mu \in \Pu(\QG)$ such that
\begin{equation} \label{Lmu}
L(x) = (\mu \ot \id)(\Ww^* (1 \ot x) \Ww), \;\;\; x\in\Linf(\QG).
\end{equation}
\end{rlist}
In the above case we in fact have $a = (\mu \ot \id)(\Ww)$. Note
further that the formula \eqref{Lmu} defines a normal unital
completely positive map on $\Linf(\QG)$ for any $\mu \in \Pu(\QG)$, to
be denoted $L_\mu$ in what follows; the correspondence $\mu \mapsto
L_{\mu}$ is injective.
\end{tw}

We need some more properties of the maps described in the above theorem, which we will call \emph{left multipliers of $\Linf(\QG)$}.

\begin{propn} \label{Matt2}
Let $\mu \in \Pu(\QG)$. The map $L_\mu$ defined above preserves the
left Haar weight $\phi$. Moreover, it restricts to a completely positive
nondegenerate \textup{(}in the sense of \cite{SSIdem}\textup{)}
map on $\C_0(\QG)$. If we consider the \textup{(}completely positive,
nondegenerate\textup{)}
map $L^u_\mu: \C_0^u(\QG) \to \C_0^u(\QG)$ given by $L^u_\mu = (\mu \ot \id)\circ
\Com_u$,
then we have
\begin{equation}
\Lambda \circ L_\mu^u = L_\mu \circ \Lambda. \label{reducingrelation}
\end{equation}
Finally if $\nu \in \Pu(\QG)$ is another state, then
\begin{equation}
L_\mu \circ L_\nu =L_{\nu \star \mu}.\label{convolutioncomposition}
\end{equation}
\end{propn}
\begin{proof}
The first fact stated above is Lemma 3.4 of \cite{KNR}. The second is
noted in \cite{dfsw} and also in \cite{dawscategory}. The third is a
consequence of the following formula, established in Proposition 6.2
of \cite{ku}:
\[ (\id \ot \Lambda) \circ \Com_u(x) = \Ww^*( \id\ot \Lambda(x))\Ww,\;\;\; x \in \C_0^u(\QG)\]
(note that the maps $\theta$ appearing in \cite{ku} disappear here as we tacitly assume that $\C_0(\hQG)$ is represented on $\mathrm{L}^2(\QG)$). Finally the last relation is easily checked on the level of the `universal' maps $L_\mu^u$ and then follows by the equality \eqref{reducingrelation}, normality of the multipliers and weak$^*$-density of $\C_0(\QG)$ in $\Linf(\G)$.
\end{proof}

Given an element $\nu \in \C_0(\G)^*$ we will write simply $L_\nu$ for
an operator formally defined as $L_{\nu \circ \Lambda}$; it is easy to
see that then $L_\nu= (\nu \ot \id)\circ\Com$
on $\C_0(\G)$ (also on $\Linf(\G)$ with a suitable
interpretation of the right-hand side).

Finally note that we can of course consider the corresponding
(unital completely positive) \emph{right multipliers}, to be denoted
$R_\mu$ for $\mu \in \Pu(\G)$. Then the left multipliers commute
with the right ones, as can be seen for example from condition (i)
in Theorem \ref{Matt} and the fact that products of elements in
$\Lone(\G)$ are dense in $ \Lone(\G)$.

\section{Proofs of the main results}

An important step towards the main theorem is
Theorem \ref{thm:weightpreserv}, which we will prove first.
To this end, we need a few lemmas.

Recall that the modular elements $\delta$ and $\delta_u$
are unbounded, strictly positive operators affiliated with
$\C_0(\G)$ and $\C_0\sp u(\G)$, respectively.

\begin{lem}
Let $\omega\in \Pu(\G)$ be an idempotent state. Then for every $t\in\br$ and $a\in \C_0^u(\QG)$
\[L^u_\omega(\delta_u^{it}a) = \delta_u^{it} L_\omega\sp u(a)
\qquad\text{and}\qquad L^u_\omega(a\delta_u^{it}) = L^u_\omega(a)\delta_u^{it}.
  \]
\end{lem}

\begin{proof}
 We prove only the first identity, the second being similar.
As $\omega$ is an idempotent, we have
\[
\omega(\delta_u^{it}) = (\omega\conv\omega)(\delta_u^{it})
= (\omega\ot\omega)(\Com_u(\delta_u^{it}))
= \bigl(\omega(\delta_u^{it})\bigr)^2
\]
(note that we need to use the strict extension of $\omega$ to $M(C_0\sp u(\G)$).
Hence $\omega(\delta_u^{it})$ is either $0$ or~$1$.
However the map $t\mapsto \omega(\delta_u^{it})$
is continuous and $\omega(\delta_u^{i0}) =\omega(1) = 1$,
so $\omega(\delta_u^{it}) = 1$ for every $t\in\br$.

Next we note that
\[
\omega(\delta_u^{it}\delta_u^{-it}) = \omega(1) = 1
= \omega(\delta_u^{it})\omega(\delta_u^{-it}).
\]
It follows that  $\delta_u^{it}$ is in the multiplicative domain
of $\omega$ for every $t\in \br$.
Then $\delta_u^{it}\ot\delta_u^{it}$ is in the multiplicative domain of
$\omega \ot \id: M(\C_0^u (\G)\ot\C_0^u (\G))  \to M(\C_0\sp u(\G))$.
Therefore
\begin{align*}
L^u_\omega(\delta_u^{it}a) = (\omega \ot \id)(\Com_u(\delta_u^{it}a))
  &= (\omega\ot\id)((\delta_u^{it}\ot\delta_u^{it}) \Com_u(a))\\
&= \bigl((\omega \ot \id)(\delta_u^{it}\ot\delta_u^{it})\bigr)
   \bigl( (\omega \ot \id)(\Com_u(a))\bigr)
  = \delta_u^{it}L^u_\omega(a).
\end{align*}
\end{proof}

\begin{cor} \label{lemma:L-omega-commutes}
Let $\omega\in \Pu(\G)$ be an idempotent state. Then for every
$t\in\br$ and $a\in \C_0(\QG)$
\begin{equation} \label{eq:L-omega commutes}
L_\omega(\delta^{it} a) = \delta^{it} L_\omega(a)
\qquad\text{and}\qquad L_\omega(a\delta^{it}) =
L_\omega(a)\delta^{it}.
\end{equation}
\end{cor}
\begin{proof}
An immediate consequence of the previous lemma, the intertwining
relation \eqref{reducingrelation} and the fact that
$\Lambda(\delta_u^{it}) = \delta^{it}$ for any $t \in \br$ (see
\cite{ku}).
\end{proof}

\begin{lem} \label{deltabehaviour}
  Let $\omega\in \Pu(\QG)$ be an idempotent state.
  For every
  \[
a\in D_\phi := \{a\in \C_0(\QG)\mid a\delta^{1/2} \text{ is bounded
  and the closure }
          \overline{a\delta^{1/2}}\in \Nphi\}
\]
we have
  \begin{equation} \label{eq:L-omega commutes2}
    L_\omega((a\delta^{1/2})^*\overline{a\delta^{1/2}}) =
   \overline{ \delta^{1/2} L_\omega(a^*a) \delta^{1/2}}.
  \end{equation}
\end{lem}

\begin{proof}
Fix $a \in D_\phi$. As $a\delta^{1/2}$ is bounded,
  also $\delta^{1/2}a^*$ is bounded and equal to $(a\delta^{1/2})^*$
  (by Corollary 8.35 of \cite{kus:weights}).
  Hence
  \[
  \delta^{1/2}a^*a\delta^{1/2}
  \]
  is bounded. By Proposition 9.24 of \cite{stratila-zsido}
  (applied to $A = \delta$, $B=\delta\inv$),
  the map
  \[
  it\mapsto \delta^{it}a^*a\delta^{it}
  \]
  has wo-continuous extension to the strip
  $\mathcal{S} := \{z\in\bc:  0\leq\textup{Re} z\leq 1/2\}$ that is analytic
  on the interior of $\mathcal{S}$.
  By Corollary~\ref{lemma:L-omega-commutes},
  \begin{equation} \label{eq:comm}
  L_\omega(\delta^{it}a^*a \delta^{it}) = \delta^{it} L_\omega(a^*a) \delta^{it}.
  \end{equation}
  We intend to apply Proposition 9.24 of \cite{stratila-zsido}
  again, this time to the function defined by the right-hand side of
  \eqref{eq:comm}. We need to show that
  the map
  \[
  it\mapsto \delta^{it} L_\omega(a^*a) \delta^{it}.
  \]
  has a wo-continuous extension to $\mathcal{S}$ that is analytic on
  the interior. Indeed by \eqref{eq:comm}
  it is enough to show that $L_\omega$ is wo-continuous
  (analyticity is clear as $L_\omega$ is bounded).
  We may restrict our considerations to $\Linf(\QG)$
  because (the closure of) $\delta^{z}a^*a \delta^{z}$
  is in $\Linf(\QG)$ for $z\in \mathcal{S}$. As   $\Linf(\QG)$ is in the
  standard form on $\Ltwo(\QG)$,
  a map on $\Linf(\QG)$ is wo-continuous if and only if
  it is normal. But $L_\omega$ is normal by Theorem \ref{Matt}.
  Hence Proposition 9.24 of \cite{stratila-zsido}
  is applicable and shows that the operator
  \[
  \delta^{1/2} L_\omega(a^*a) \delta^{1/2}
  \]
  is bounded (and densely defined as its domain is a core of
  $\delta^{1/2}$).
   The claimed equality follows from \eqref{eq:comm}
  and the uniqueness of analytic extensions.
\end{proof}

\begin{proof}[Proof of Theorem \ref{thm:weightpreserv}]
Let $\omega \in \Pu(\QG)$ be an idempotent state.
By Proposition \ref{Matt2}, $L_\omega$ preserves
the left Haar weight. For the right Haar weight we use the fact that
on the formal level we have the equality
$\psi = \phi(\delta^{1/2}\cdot \delta^{1/2})$. Indeed, with Lemma
\ref{deltabehaviour} in hand, we can repeat the second part of the
proof of Proposition 3.12 in \cite{SSIdem} to conclude the
argument. We outline the main steps for the convenience of the reader:
first, for all $a\in D_\phi$ we have $\psi(a^*a) < \infty$ and
by~\eqref{eq:L-omega commutes2} also  $\psi(L_\omega(a^*a))<\infty$
(see Corollary 8.35 of \cite{kus:weights}). Then applying $\phi$ to
equality~\eqref{eq:L-omega commutes2} gives $\psi(a^*a) = \psi(L_\omega(a^*a))$.
Finally, due to the uniqueness of the right Haar weight and density of
the elements of the form $a^*a$ for $a \in D_\phi$ in $\C_0(\QG)_+$,
we deduce that $\psi=\psi\circ L_\omega$.
\end{proof}

The main theorem is proved with the help of
Theorem \ref{thm:weightpreserv}, but we will also need a few
additional preliminary results.

\begin{lem}\label{idempexp}
Let $\mu \in \Pu(\G)$. Then $\mu$ is an idempotent state if and only
if $L_\mu$ is a conditional expectation if and only if
$L_\mu|_{\C_0(\QG)}$ is a conditional expectation.
\end{lem}
\begin{proof}
If either $L_\mu$ or $L_\mu|_{\C_0(\QG)}$ is a conditional
expectation, then $L_\mu$ is an idempotent map (in the second case by
weak$^*$-continuity). Thus applying formula
\eqref{convolutioncomposition} and injectivity of the map $\mu \mapsto
L_\mu$ shows that $\mu$ must be an idempotent state.

Assume then that $\mu \in \Pu(\G)$ is an idempotent state. We need to
show that the images of the idempotent maps $L_\mu$ and
$L_\mu|_{\C_0(\G)}$ are algebras (hence C*-algebras).
This is equivalent to showing that for any $a,b \in \Linf(\G)$
\[
L_\mu(a) L_\mu(b) = L_\mu (L_\mu(a) b)
\]
because a norm-one positive projection onto a C*-subalgebra is a
conditional expectation.
Normality of $L_\mu$ and weak$^*$-density of $\C_0(\QG)$ in
$\Linf(\G)$ mean that it suffices to establish the displayed formula
for $a, b\in \C_0(\G)$. As the reducing morphism is a surjective
$^*$-homomorphism and we have the relation \eqref{reducingrelation},
it suffices to show
\[ L_\mu^u(x) L_\mu^u(y) = L_\mu^u (L_\mu\sp u(x) y), \;\;\; x, y \in \C_0^u(\QG).\]
For that however we can follow word by word the proofs in Lemma 2.3
and Theorem 2.4 of \cite{SSIdem}, as they use only the general
properties of the coproducts and the quantum cancellation rules, which
hold also for $\C_0^u(\QG)$ (see \cite{ku}).
\end{proof}

The following observation is due to Matthew Daws. Note that the
positivity is only used to make sense of the intertwining relation
(which otherwise could be formulated in the weak sense).

\begin{lem} \label{lem::C0multipliers}
Let $Z:\C_0(\QG) \to \C_0(\QG)$ be a completely positive nondegenerate
map such that $\Com \circ Z = (Z \ot \id) \circ \Com$. Then there
exists a \textup{(}unital completely positive\textup{)} left multiplier $L$ of
$\Linf(\QG)$ such that $Z=L|_{\C_0(\G)}$.
\end{lem}

\begin{proof}
Recall that $\Lone(\G)$ is a closed, weak$^*$-dense ideal in the
Banach algebra $\C_0(\QG)^*$. The linear span
of the products of elements in $\Lone(\G)$ is dense in $\Lone(\G)$ (see
for example \cite{ds}, although the result dates back already to
\cite{KV}). The assumed intertwining relation on $Z$ implies that for
any $\omega_1, \omega_2 \in \C_0(\QG)^*$ we have $Z^*(\omega_1 \star
\omega_2) = Z^*(\omega_1) \star \omega_2$. In particular for
$\omega_1, \omega_2 \in \Lone(\G)$ we have $Z^*(\omega_1 \star
\omega_2)= Z^*(\omega_1) \star \omega_2 \in \Lone(\G)$ by the ideal
property. As $Z^*$ is bounded, it means $Z^*:\Lone(\QG) \to
\Lone(\QG)$ and we can consider $L:=(Z^*|_{\Lone(\QG)})^*$. It is
clear that $L:\Linf(\QG) \to \Linf(\QG)$ is a normal map, and by
weak$^*$-density of $\Lone(\G)$ in $\C_0(\QG)^*$ it follows that
$L|_{\C_0(\QG)}= Z$. This means in particular that $L$ is unital and
completely positive. Now the fact that $L$ is a left multiplier of
$\Linf(\QG)$ follows by condition (i) in Theorem \ref{Matt}, as $Z^*$
was a left module map.
\end{proof}

\begin{propn} \label{backcst}
Let $\Clg$ be a right-invariant $\psi$-expected
C*-subalgebra of $\C_0(\QG)$. Then there exists
an idempotent state $\omega \in \Pu(\G)$ such that
$\Clg = L_\omega(\C_0(\QG))$.
\end{propn}

\begin{proof}
Let $E$ be the $\psi$-preserving conditional expectation onto $\Clg$.
Arguing as in Proposition~3.3 of \cite{SSIdem}, we obtain equality
$E R_\nu = R_\nu E$ for all $\nu \in \Lonesharp(\QG)\subset
\C_0(\QG)^*$. Then the argument in the proof of the implication
(iii)$\Longrightarrow$(ii) of Lemma 1.6 in \cite{SSIdem} yields the
commutation relation \eqref{commutCom} on $\C_0(\QG)$. Lemma
\ref{lem::C0multipliers} implies that $E$ is a restriction of
a left multiplier of $\Linf(\QG)$, i.e.\ a map of the form $L_\omega$
for some $\omega \in \Pu(\QG)$.  Finally, by Lemma \ref{idempexp},
$\omega$ must be an idempotent state.
\end{proof}

The next result is an analog of the last one in the context of von
Neumann subalgebras of $\Linf(\QG)$; it is much simpler.

\begin{propn} \label{backvNa}
Let $\Dlg$ be a right-invariant $\psi$-expected von Neumann subalgebra
of $\Linf(\QG)$. Then there exists an idempotent state $\omega \in
\Pu(\G)$ such that $\Dlg = L_\omega(\Linf(\QG))$.
\end{propn}

\begin{proof}
Let $E$ be the $\psi$-preserving conditional expectation onto $\Clg$.
Arguing as in Proposition~3.3 of \cite{SSIdem}, we obtain equality
$E R_\nu = R_\nu E$ for all $\nu \in \Lonesharp(\QG)$.  Since
$\Lonesharp(\QG)$ is dense in $\Lone(\QG)$, it follows
that $E_*: \Lone(\QG)\to \Lone(\QG)$ is a left module map.
Hence, by Theorem \ref{Matt}, $E = L_\omega$ for some $\omega \in\Pu(\G)$.
Once again, by Lemma \ref{idempexp}, $\omega$ must be an idempotent state.
\end{proof}

\begin{proof}[Proof of Theorem \ref{maintheorem}]
Given an idempotent state $\omega \in \Pu(\G)$, we associate to it the
algebras $\Clg:=L_\omega(\C_0(\QG))$ and
$\Dlg:=L_\omega(\Linf(\QG))$. They are respectively a nondegenerate
C*-subalgebra of $\C_0(\G)$ and a von Neumann subalgebra of
$\Linf(\G)$. As left multipliers commute with right multipliers,
both $\Clg$ and $\Dlg$ are right invariant in the appropriate
sense. By Lemma \ref{idempexp} the map $L_\omega$ is a normal
conditional expectation
onto $\Dlg$ and when restricted to $\C_0(\G)$ a conditional
expectation onto $\Clg$. By Theorem \ref{thm:weightpreserv} it
preserves both the left and the right Haar weight. This means that to
any idempotent state we can associate a right-invariant expected
C*-subalgebra of $\C_0(\QG)$ and a
right-invariant  expected von Neumann subalgebra of $\Linf(\QG)$.

Conversely, if we are given a right-invariant expected
C*-subalgebra $\Clg$ of $\C_0(\QG)$, Proposition \ref{backcst} shows
that the relevant conditional expectation is of the form $L_\omega$
for some idempotent state $\omega \in \Pu(\G)$.
Proposition \ref{backvNa} gives a similar implication in the
$\Linf(\QG)$-case.

The fact that these correspondences are bijective follows from the
uniqueness of conditional expectations preserving faithful weights and
the injectivity of the map $\mu \mapsto L_\mu$.

This establishes the bijections between objects in (i), (ii) and
(iii). It is clear that working with right multipliers we could
establish similarly bijections between objects in (i), (iv) and~(v).
\end{proof}

\begin{proof}[Proof of Corollary \ref{corollary}]
Given a right-invariant $\psi$-expected C*-subalgebra $\Clg$ of
$\C_0(\QG)$, we see from Proposition \ref{backcst} that it is of the
form $L_\omega(\C_0(\QG))$ for some idempotent state $\omega \in
\Pu(\G)$. Then Theorem \ref{maintheorem} and its proof show that
$\Clg$ is indeed expected. The von Neumann algebra argument is
identical, once we use Proposition \ref{backvNa}.
\end{proof}

\begin{proof}[Proof of Proposition \ref{scaling}]
We begin with the scaling group. Proposition 9.2(2) of \cite{ku} shows
that we have the following equality for any $t \in \br$:
\[
\Com_u \circ \sigma_t^u = (\tau_t^u \ot \sigma_t^u ) \circ \Com_u.
\]
This easily implies that for any $\mu \in \M\sp u(\G)$ and $t \in \br$
we have
\[ L_{\mu \circ \tau_t\sp u}^u =  \sigma_{-t}^u \circ L_\mu^u \circ \sigma^u_t.\]
Applying the reducing morphism and using the
relation $\sigma_t \circ \Lambda = \Lambda \circ \sigma_t^u$,
also shown in \cite{ku}, yields
\begin{equation} L_{\mu \circ \tau_t^u} =  \sigma_{-t} \circ L_\mu
  \circ \sigma_t, \qquad t \in \br.  \label{intertwine}\end{equation}
For an idempotent state $\omega \in \Pu(\G)$, the conditional
expectation $L_\omega$ preserves
the left Haar weight, so by Takesaki's theorem $L_\omega$ commutes with the
modular automorphisms $\sigma_t$. An application of the formula
\eqref{intertwine} yields the equality $ L_{\omega \circ \tau_t^u}=
L_\omega$, which by the injectivity of the map $\mu \mapsto L_\mu$
implies that $\omega = \omega \circ \tau_t^u$ for all $t \in \br$.

The proof that $\omega=\omega \circ S^u$ follows exactly as in
Proposition 2.6 of \cite{SSIdem},  working all the time with the
semi-universal unitary  and using in the last step
Corollary 9.2 of \cite{ku}, which describes the core of $S^u$. Finally
the fact that $\omega$ preserves the unitary antipode follows as the
latter can be expressed as (extension of) the composition of $S^u$ and
$\tau^u_{i/2}$.
\end{proof}

In the following proof, we will use certain results of
\cite{KK}. Apparent differences in terminology stem from the different
choice of conventions regarding multiplicative unitaries in that
paper.

\begin{proof}[Proof of Theorem \ref{Haar}]
(i)$\Longrightarrow$(ii)
This follows as in the proof of the analogous implication of Theorem
3.7 of \cite{SSIdem}, using the fact that the null space of the Haar
state of a compact quantum group is a two-sided ideal, observed
already in \cite{Wor}.

(ii)$\Longrightarrow$(i)
Again one can use the same ideas as in \cite{SSIdem}. We sketch the
outline of the argument: one considers the new C*-algebra
$\blg:=\C_0^u(\QG)/{N_\omega}$ with the canonical quotient map
$\pi:\C_0^u(\QG) \to \blg$. Choosing an element $e \in \C_0^u(\QG)_+$
such that $\omega(e)=1$, we deduce first that $e$ is in the
multiplicative domain of the map $L_\omega^u$ (using the proof of
Lemma \ref{idempexp}) and then that $\pi(L_\omega(e))$ is a unit of
$\blg$. In the next step we establish the fact that there exists a
faithful state $\mu \in \blg^*$ such that $\omega = \mu \circ \pi$.
This in turn implies that there is a well-defined map
$\Delta_\blg: \blg \to \blg \ot \blg$ such that
\[ \Delta_\blg \circ \pi = (\pi \ot \pi) \circ \Com.\]
Since the quantum cancellation properties hold on the $\C_0^u(\QG)$-level
(see \cite{ku}), it follows that the
pair $(\blg, \Com_\blg)$ satisfies the Woronowicz's axioms (with $\mu$
playing the role of the bi-invariant state). Therefore there exists a
compact quantum group $\QH$ such that $\blg = \C(\QH)$. It is then
fairly standard to check (see Proposition 3.5 in \cite{dawsBohr}) that
the map $\pi$ lifts to a quantum group morphism
$\pi^u:\C_0^u(\QG) \to \C^u(\QH)$ (which will still be surjective by
Theorem~3.6 of \cite{DKSS}). Hence $\omega = h_\QH \circ \pi^u$ is a
Haar idempotent.

(i)$\Longrightarrow$(iii)
  If $\omega$ is a Haar idempotent, then
  the right-hand version of Theorem 5.14 of \cite{KK} shows that
the algebra $\Clg^u:=L_\omega^u(\C_0^u(\QG))$ satisfies the universal
symmetry condition, i.e.\
\begin{equation} \label{symmetryuniv}
  \Ww^*(\Clg\ot 1) \Ww  \subset \M(\Clg^u \ot C_0(\hQG)).
\end{equation}
Applying to the inclusion above the reducing morphism $\Lambda$
tensored by the identity and noting that $L_\omega(\C_0(\QG))=
\Lambda(L_\omega^u(\C_0^u(\QG)))$ shows that $L_\omega(\C_0(\QG))$ is
symmetric.

(iii)$\Longrightarrow$(i)
This is essentially contained in the proof of Theorem 5.15 in
\cite{KK}: one first uses Proposition 3.13 of that paper to show that
$\Clg^u:=L_\omega^u(\C_0^u(\QG))$ satisfies the universal symmetry
condition \eqref{symmetryuniv} and then applies Theorem~5.14 of
\cite{KK} to deduce that $\omega$ is a Haar idempotent.
\end{proof}

Finally note that (as mentioned in the proof) Theorem 5.14 of
\cite{KK} provides a version of the correspondence
(i)$\iff$(iii) valid on the universal level.
The coamenable version of this correspondence was first noted in
\cite{Pekka} (with compact quantum subgroups instead of Haar
idempotents).

\end{document}